\documentclass[12pt,reqno]{amsart}
\usepackage{amsmath, amssymb, amsbsy, amsfonts, amsthm, latexsym, amsopn, amstext, amsxtra, euscript, amscd, color, mathrsfs}
\usepackage{amsmath,amscd}
\usepackage[normalem]{ulem}
\usepackage{soul}

\usepackage[
  a4paper,
  left=4.5cm,
  right=1.4cm,
  top=2cm,
  bottom=2cm,
  twoside=false
]{geometry}

\makeatletter
\@namedef{subjclassname@2020}{%
  \textup{2020} Mathematics Subject Classification}
\makeatother
\PassOptionsToPackage{hyphens}{url}
\usepackage[hidelinks]{hyperref}

\usepackage{cite}
\usepackage{tabularx}
\usepackage[colorinlistoftodos,prependcaption,textsize=tiny]{todonotes}

\usepackage{soul}
\usepackage{tikz}

\usepackage{amscd}
\usepackage{color,enumerate}
\usepackage{array, multirow, multicol}
\usepackage{longtable}
\usepackage{bm}  
\usepackage[hidelinks]{hyperref}
\newcommand{\RNum}[1]{\lowercase\expandafter{\romannumeral #1\relax}}

\newtheorem{thm}{Theorem}[section]
\newtheorem{lem}[thm]{Lemma}

\newtheorem{prop}[thm]{Proposition}

\newtheorem{rmk}[thm]{Remark}

\newtheorem{thm-con}[thm]{Theorem-Conjecture}
\numberwithin{equation}{section}

\theoremstyle{definition}
\newtheorem{defn}[thm]{Definition}

\DeclareMathOperator{\lcm}{lcm}

\newcommand{\F}{\mathbb F}

\def\Tr{{\rm Tr}}

\begin{document}
\title[Some new results on permutation trinomials]{Some new results on permutation trinomials  over finite fields with even characteristic}

\author[K. Garg]{Kirpa Garg}
\address{Department of Computer Science, University of Rouen Normandy, Rouen 76000, France}
\email{kirpa.garg@gmail.com}
\author[S. U. Hasan]{Sartaj Ul Hasan}
\address{Department of Mathematics, Indian Institute of Technology Jammu, Jammu 181221, India}
\email{sartaj.hasan@iitjammu.ac.in}
\author[C. K. Vishwakarma]{Chandan Kumar Vishwakarma}
\address{Department of Mathematics, Indian Institute of Information Technology Bhopal, Bhopal 462003, India}
\email{7091chandankvis@gmail.com}

\thanks{The work of S. U. Hasan and C. K. Vishwakarma is partially supported by the Science and Engineering Research Board (SERB), Government of India, under grant CRG/2022/005418.}

\begin{abstract}
The construction of permutation trinomials of the form
$X^r(X^{\alpha (2^m-1)}+X^{\beta(2^m-1)} + 1)$ over $\F_{2^{2m}}$, where $m,~r\text{ and }\alpha > \beta$ are positive integers, is an active area of research.  Several classes of permutation trinomials with fixed values of $\alpha$, $\beta$ and $r$ have been studied. Here, we construct three new classes of permutation trinomials with $(\alpha,\beta,r)\in\{(7,5,7),(8,6,9),(10,4,11)\}$ over $\F_{2^{2m}}$. We also analyze the quasi-multiplicative equivalence of the newly obtained classes of permutation trinomials to both the existing ones and to each other. Additionally, we prove the nonexistence of a class of permutation trinomials over $\F_{2^{2m}}$ of the same type for $r=9$, $\alpha=7$, and $\beta=3$ when $m > 3$. Furthermore, we provide a proof for a conjecture on the quasi-multiplicative equivalence of two classes of permutation trinomials, as proposed by Yadav, Gupta, Singh, and Yadav (2024). 
\end{abstract}
\maketitle

\section{Introduction}
Let $\mathbb{F}_{q}$ be the finite field with $q$ elements. Throughout this paper, we assume $q=2^m$, where $m$ is a positive integer, unless stated otherwise. We denote by $\F_q^*$ the multiplicative cyclic group of non-zero elements of the finite field $\F_q$ and by $\F_q[X]$ the ring of polynomials in one variable $X$ with coefficients in $\F_q$.  A polynomial $f(X)$ in $\mathbb{F}_{q}[X]$ is called a permutation polynomial of $\mathbb{F}_{q}$ if its associated mapping $X \mapsto f(X)$  is a permutation of $\mathbb{F}_{q}$. Permutation polynomials over finite fields have been a significant area of research for many years due to their broad applications in various fields of engineering and mathematics, such as cryptography, coding theory, and combinatorial design theory. For a survey of recent developments and contributions in this area, the reader is referred to~\cite{hou2015permutation,li2019survey,wang2024survey} and the references therein.

There has been significant interest in permutation polynomials, particularly those with nice algebraic forms and additional desirable properties. While a random permutation polynomial can be readily constructed using Lagrange interpolation, such permutation polynomials rarely exhibit a concise or structured form. Among permutation polynomials, monomials are the simplest and can be easily classified. In contrast, the classification of permutation binomials and trinomials presents a greater challenge. For explicit constructions of permutation binomials and trinomials, the reader is referred to \cite{gupta2022new,RG, hou2015survey, ZLF, wu2017permutation, KLQ2018, hou2015permutation, Li, NT2017, NT2019, ozbudak2023complete, LHJ2020}.

It is well known that permutation binomials with identity coefficients do not exist over finite fields of even characteristic. Motivated by this, several researchers have investigated classes of permutation trinomials of the form $X^{r}(X^{\alpha(q-1)}+X^{\beta(q-1)} + 1 )$ over finite fields of even characteristic; see, for example,~\cite{RG, Harsh, Li, ZLF} and the references therein. In 2016, Gupta and Sharma~\cite{RG} introduced four new classes of such permutation trinomials with $(\alpha, \beta, r)\in\{(3,2,2),(3,1,4),(4,3,3),(4,1,5)\}$. Moroever, additional classes of permutation trinomials of the same form were proposed in~\cite{Li, ZLF} for $(\alpha,\beta,r)\in\{(3,1,3),(3,2,3),(3,1,4),(4,3,5),(5,1,4),(5,2,5),(5,4,5),(6,2,5)\}$. Subsequently, Yadav, Gupta, Singh, and Yadav~\cite{Harsh} identified three new classes of permutation trinomials of this type with $(\alpha,\beta, r)\in\{(6,5,5),(6,1,7),(6,4,7)\}$. In this paper, we extend this investigation by  exploring additional classes of permutation trinomials of this form  with $(\alpha,\beta,r)\in\{(7,5,7),(8,6,9),(10,4,11)\}$.

In this work, the problem of analyzing the permutation behavior of the proposed trinomial is first converted into studying a polynomial in two variables over the unit circle. Then, an algorithmic approach is applied to factorize this polynomial, as its factors play a key role in determining whether the proposed trinomials exhibit permutation behavior. On the other hand, to establish the nonexistence of a certain class of permutation trinomials over $\F_{q^{2}}$, we leverage the connection between permutation polynomials and algebraic curves over finite fields. Specifically, it turns out that the Hasse-Weil bound serves as a powerful tool for proving the nonexistence of a permutation polynomial over $\F_{q}$ when $q$ is sufficiently large. However, a significant technical challenge in applying the Hasse-Weil bound lies in proving the absolute irreducibility of the algebraic curve involved.  Several researchers~\cite{BB, BM, Hou, HC} have demonstrated the nonexistence of permutation polynomials by showing that the corresponding curve is absolutely irreducible over $\F_{q}$. Here, we demonstrate the absolute irreducibility of a specific curve over $\F_q$ and, using the Hasse-Weil bound, identify a class of trinomials that does not permute $\F_{q^2}$.

A permutation polynomial $f(X)$ over $\mathbb{F}_{q}$ remains a permutation polynomial when transformed into $af(bX + c) + d$ for $a,b\in\F_{q}^{*}$ and $c,d\in\F_{q}$. While this transformation preserves the degree of the polynomial, it may alter the number of monomials. To overcome this limitation, Wu, Yuan, Ding, and Ma~\cite{wu2017permutation} introduced the concept of Quasi-Multiplicative (QM) equivalence, which ensures that the number of monomials remains unchanged. Identifying QM equivalence between two classes of permutation polynomials is a challenging problem. In this paper, we verify that the proposed trinomials are QM inequivalent both among themselves and to those previously known in the literature.  In 2024, Yadav, Gupta, Singh, and Yadav~\cite{Harsh} proposed a conjecture regarding the QM equivalence of two classes of permutation trinomials over quadratic extensions of finite fields with arbitrary characteristic. We provide a different yet elementary proof of their conjecture, which appears to have been resolved by Yadav, Singh, and Gupta~\cite{Yadav}.

We shall now give the structure of the paper. We first recall some results in Section ~\ref{S2}. In Section ~\ref{S3}, we present three new classes of permutation trinomials over the finite field $\F_{q^{2}}$. Section~\ref{S4} deals with the QM equivalence of the permutation trinomials introduced in this paper to known ones, confirming that the classes we obtain are indeed new. Additionally, in Section~\ref{S5}, we show that a family of trinomials $X^{9}(X^{7(q-1)}+X^{3(q-1)}+1)$ does not permute $\mathbb{F}_{2^{2m}}$ for $m>3$. Furthermore, in Section~\ref{S6}, we give a proof of the conjecture on the QM equivalence of trinomials stated in~\cite{Harsh}. Finally, we conclude the paper in Section~\ref{S7}.

\section{Preliminaries}\label{S2}
In this section, we recall some results to be used later in subsequent sections. Throughout the paper, we denote by 
$\mbox{Tr}_m^n$ the (relative) trace function from $\F_{p^n} \rightarrow \F_{p^m}$, i.e., $ \mbox{Tr}_m^
n(X)=\sum^{\frac{n-m}{m}}_{i=0}X^{p^{mi}},$ where $m$ and $n$ are positive integers with $m \mid n$, and $p$ is a prime number.

Note that the results stated below in this section are true for an arbitrary prime power $q = p^n$, where $n$ is an integer and $p$ is a prime number. We now present a well-known criterion studied in various forms by several authors; for instance, Wan-Lidl~\cite{Wan}, Park-Lee~\cite{Park}, Akbary-Wang~\cite{AW2007}, Wang~\cite{Wang}, and Zieve~\cite{Zieve}. This method is particularly helpful for determining whether a polynomial of the form $X^r h(X^{s})$ permutes $\F_{q}$. 
The criterion is stated in the following lemma.
\begin{lem}
Let $h(X)\in \F_{q}[X]$ and $r, s$ are positive integers such that $s\mid (q-1)$. Then $f(X)=X^r h(X^{s})$ permutes $\F_{q}$ if and only if the following two conditions
\begin{enumerate}
\item $\gcd(r, s) = 1$,
\item $ X^r ({ h(X) })^{s}$ permutes $\mu_{\frac{q-1}{s}}$,
\end{enumerate}
 are satisfied, where $\mu_{\frac{q-1}{s}}=\{a\in \F_{q}^{*} \mid a^{\frac{q-1}{s}}=1\}$ is the unit circle of order $\frac{q-1}{s}$ in $\F_{q}$.
\end{lem}

The special case of the above lemma for the finite field $\F_{q^2}$ and $s=q-1$ gives the following result, which we use to determine whether the new families of trinomials considered in this paper permute $\F_{q^2}$.

\begin{lem}\label{PP}
For $h(X) \in \F_{q^2} [X]$, $f(X) = X^r h(X^{q-1})$ permutes $\F_{q^2}$ if and only if the following two conditions hold
\begin{enumerate}
\item $\gcd(r, q - 1) = 1$,
\item $g(X) = X^r(h(X))^{q-1}$ permutes $\mu_{q+1}$.
\end{enumerate}
\end{lem} 

The following lemma presents a result by Yadav, Gupta, Singh, and Yadav \cite{Harsh}, which we use throughout the paper.
\begin{lem}\cite[Lemma 2.2]{Harsh}{\label{roots}}
Let $q$ be any prime power co-prime to 3 and $\alpha,\beta \in \mathbb{N}$ be such that $\alpha>\beta$. Then, the polynomials $X^\alpha+ X^\beta + 1$ and $X^\alpha+X^{(\alpha-\beta)}+1$ have no roots in $\mu_{q+1}$ if any one of the following conditions holds
\begin{enumerate}
\item $\gcd(\alpha+\beta,q+1)=1$,
\item $\gcd(|\alpha-2\beta|,q+1)=1$,
\item $\gcd(2\alpha-\beta,q+1)=1$,
\item  $\gcd(3,q+1)=1$,
\end{enumerate}
where $| \cdot |$ denotes the usual absolute value.
\end{lem}

We also recall the Chinese Remainder Theorem, which will be needed later.
\begin{lem}[Chinese Remainder Theorem (CRT)]\label{CRT}
    Let \( a_1, a_2, \dots, a_t \) be any \( t \) integers, and let \( m_1, m_2, \dots, m_t \) be pairwise coprime positive integers. Then, there exists an integer \( X \) such that  
    \[
    X \equiv a_i \pmod{m_i}, \quad \text{for } i = 1,2,\dots,t.
    \]
    Moreover, the solution \( X \) is unique modulo \( M = m_1 m_2 \cdots m_t \).
\end{lem}

\section{Three new classes of permutation trinomials}\label{S3}
In this section, we introduce three new classes of permutation trinomials over the finite field $\F_{q^2}$. We begin with the following lemma, which is used to determine our first family of permutation trinomials. 
\begin{lem}\label{Tr}
Let $m$ be a positive integer such that $\gcd(5,m)=1$. Let $b\in\mathbb{F}_{2^{10m}}$ be such that
$b^{10}+b^6+b^5+b^3+b^2+b+1=0.$ Then, we have
$\operatorname{Tr}_{2m}^{10m}(b)=\operatorname{Tr}_{2m}^{10m}(b^{33})=0.$
\end{lem}
 \begin{proof} Since $\gcd(5,m)=1$ and $b$ satisfies the irreducible polynomial $X^{10}+X^6+X^5+X^3+X^2+X+1$ over $\F_2$, it follows that $b\not\in \F_{2^{2m}}$ and $b\in\mathbb{F}_{2^{10}}\subseteq\mathbb{F}_{2^{10m}}$. In addition, we can write $m=5k+r$ for some integers $k\geq 1$ and $0< r<5$. This implies that
\begin{equation}{\label{trace}}
\operatorname{Tr}_{2m}^{10m}(b)=b+b^{2^{2r}}+b^{2^{4r}}+b^{2^{{6r}}}+b^{2^{{8r}}}=b+b^{2^2}+b^{2^4}+b^{2^{6}}+b^{2^{8}}.
\end{equation}
Now, using the relations $b^{10} + b^6 + b^5 + b^3 + b^2 + b + 1=0 $ and $b^{2^{10}}=b$, we compute\
\begin{equation*}
\begin{tabular}{cc}
$\begin{array}{l}
b^{2^4}=b+b^5+b^7+b^9,\\
b^{2^6}=b^2+b^4+b^5+b^6+b^7+b^9,\\
b^{2^8}=b^2+b^6.		
\end{array}$
\end{tabular}
\end{equation*}
Substituting the obtained values of $b^{2^i}$ for $i \in \{4,6,8\}$ in Equation \eqref{trace}, we get
$$\operatorname{Tr}_{2m}^{10m}(b)=b+b^{2^2}+b^{2^4}+b^{2^{6}}+b^{2^{8}}=0.$$
Consequently, we get
\begin{align*}
\operatorname{Tr}_{2m}^{10m}(b^{33})&={bb^{2^{5}}+{(bb^{2^5})}^{2^2}+{(bb^{2^5})}^{2^4}+{(bb^{2^5})}^{2^6}+{(bb^{2^5})}^{2^8}}\\&=bb^{2^5}+b^{2^2}b^{2^7}+b^{2^4}b^{2^9}+b^{2^6}b^2+b^{2^8}b^{2^3}=0.
\end{align*}
\end{proof}

We now introduce our first class of permutation trinomials, where $r=11$, $\alpha=10$ and $\beta=4$.
\begin{thm}\label{T1}
    Let \( q = 2^m \). Then, the trinomial $F_{1}(X) = X^{11} \left( X^{10(q-1)} + X^{4(q-1)} + 1 \right)$ permutes \( \mathbb{F}_{q^2} \) if and only if \( m \not\equiv 0 \pmod{5} \).
\end{thm}
\begin{proof} We first assume that $m\not\equiv0\pmod{5}$. From Lemma~\ref{PP}, for $F_1$ to be a permutation polynomial over $\F_{q^2}$, it is sufficient to show that $\gcd(11, q-1) = 1$ and $G_1(X)=X^{11}(X^{10}+X^4+1)^{q-1}$ permutes $\mu_{q+1}$.  Since $m\not\equiv 0\pmod{5}$,
it is straightforward to observe that $\gcd(11,q-1)=1$.
Additionally, by Lemma~\ref{roots}, the polynomial $X^{10}+X^{4} +1$ has no roots in $\mu_{q+1}$. Therefore, we can write $G_1(X)=\dfrac{X^{11}+X^7+X}{X^{10}+X^4+1}$ over $\mu_{q+1}$. Now, to prove that $G_1(X)$ permutes $\mu_{q+1}$, we show that $G_1(X)$ is injective on $\mu_{q+1}$. Suppose that $G_1(X)=G_1(Y)$ for some $X,Y \in \mu_{q+1}$. This leads us to the following equation
\begin{equation*}
    (X^{11}+X^7+X)(Y^{10}+Y^4+1)+(Y^{11}+Y^7+Y)(X^{10}+X^4+1)=0.
\end{equation*}
We denote the left-hand side of the above equation as $f(X, Y)$. By using SageMath, we get that $f(X,Y)$ can be factored into the product of six polynomials over the quintic extension $\F_{(q^2)^{5}}$ of $\F_{q^2}$ as following
$$f(X, Y)=(X + Y)f_1(X, Y)f_2(X, Y)f_3(X, Y)f_4(X, Y)f_5(X, Y), \mbox{~~where}$$ 
\begin{align*}
&f_1(X,Y)=X^2Y^2+b^{33}(X^2+Y^2)+b^{528}XY+1, \\&
f_2(X,Y)=X^2Y^2+b^{66}(X^2+Y^2)+b^{33}XY+1, \\& 
f_3(X,Y)=X^2Y^2+b^{132}(X^2+Y^2)+b^{66}XY+1, \\&
f_4(X,Y)=X^2Y^2+b^{264}(X^2+Y^2)+b^{132}XY+1, \\&
f_5(X, Y)=X^2Y^2+b^{528}(X^2+Y^2)+b^{264}XY+1,  
\end{align*}
and $b \in \F_{2^{10}} \subset \F_{q^{10}}$ is a $(2^{10}-1)$th primitive root of unity satisfying $X^{10} + X^6 + X^5 + X^3 + X^2 + X + 1 $. Since $m \not \equiv 0 \pmod 5$, it follows that $b \not \in \F_{q^2}$ but $b \in \F_{q^{10}}$. 

Clearly, $f(X, Y) = 0$ if and only if either $X=Y$ or $f_i(X,Y)=0$ for some $1 \leq i \leq 5$ in $\mathbb{F}_{q^{10}}$. First, suppose that $f_{1}(X,Y)=0$, which implies $\Tr_{2m}^{10m}(f_1(X,Y))=0$. From Lemma \ref{Tr}, we obtain $\Tr_{2m}^{10m}(f_1(X,Y))=X^2Y^2+1$. Thus, we have $X^2Y^2+1=(XY+1)^2=0$ i.e., $XY=1$. Substituting $XY=1$ or $X= \frac{1}{Y}$ into the expression for $f_1(X, Y)$, we get $Y=X^q$ and this will further imply that $b^{495}\in\F_{q}.$ However, since the multiplicative order of $b^{495}$ in $\mathbb{F}_{2^{10}}^{*}$ is $31$, and $5$ is the least positive integer such that $q^{5}\equiv 1\pmod{31}$, we get that $b^{495}\in\F_{q^5}\setminus \F_{q}$, a contradiction. This disproves our assumption that $f_{1}(X,Y)=0$ for every $X,Y\in \mu_{q+1}$. Similarly, one can verify that $f_{i}(X,Y)\neq 0$ for $X,Y\in \mu_{q+1}$ where $i\in\{2,3,4,5\}$. Therefore, the only possibility for $f(X,Y)=0$ is when $X+Y=0$, which in turn implies that $G_1(X)$ is injective on $\mu_{q+1}$.
	
Conversely, let $F_1(X)$ permute $\F_{q^2}$ for some $m \equiv 0 \pmod 5$, or equivalently, for $m=5k$ where $k$ is a positive integer. If $k$ is even, then $\gcd(11, 2^m-1) \neq 1$. By Lemma~\ref{PP} this contradiction the assumption that $F_1(X)$ permutes $\F_{q^2}$. Next, if $k$ is odd, we have $q+1 \equiv 0 \pmod{11}$ and $q + 1 \equiv 0 \pmod{3}$. In this case, we can choose an element $\alpha \in \mu_{11} \subset \mu_{q+1}$ and then verify that $G_{1}(\alpha) \in \mu_3 \subset \mu_{q+1}$ for each $\alpha \in \mu_{11}$. Therefore, we can conclude that $G_{1}(X)$ does not permute $\mu_{q+1}$. Consequently, by Lemma~\ref{PP}, we get a contradiction to the assumption that $F_{1}(X)$ permutes $\mathbb{F}_{q^2}$.
\end{proof}
Next, we present another class of permutation trinomials over $\F_{q^2}$ with $r=9, \alpha=8$ and $\beta=6$.
\begin{thm}\label{T2}
Let $q=2^m$, then the trinomial $F_{2}(X)=X^{9}(X^{8(q-1)} + X^{6(q-1)} + 1)$ permutes $\F_{q^2}$ if and only if $m$ is odd.
\end{thm}
\begin{proof}
It is known from Lemma~\ref{PP} that $F_2(X)$ permutes the finite field $\F_{q^2}$ if and only if $\gcd(9, q-1) = 1$ and $G_2(X)=X^{9}(X^8+X^6+1)^{q-1}$ permutes $\mu_{q+1}$. We first assume that $m$ is odd which ensures $\gcd(9, q-1)=1$. Thus, it suffices to show that $G_2(X)$ permutes $\mu_{q+1}$. It is straightforward to verify that if $X \in \mu_{q+1}$, then $G_2(X)$ also belongs to the unit circle $\mu_{q+1}$. Since, from Lemma~\ref{roots}, $X^{8}+X^{6}+1$ has no roots in $\mu_{q+1}$, we can express $G_2(X)$ as $G_2(X)=\dfrac{X^9+X^3+X}{X^8+X^6+1}$ over $\mu_{q+1}$. Now, we proceed as in Theorem~\ref{T1} to establish the injectivity of $G_2(X)$ on $\mu_{q+1}$.
Let $G_2(X) = G_2(Y)$ for some $X,Y \in \mu_{q+1}$. Then, we obtain the following equation
\begin{align*}
& X^9Y^8+X^8Y^9+X^9Y^6+X^6Y^9+X^3Y^8+X^8Y^3+X^3Y^6+X^6Y^3
\\& \quad + XY^8  +X^8Y+XY^6+X^6Y+X^9+Y^9+X^3+Y^3+X+Y=0.
\end{align*}
If $f(X, Y)$ denotes the left-hand side of the above equation, then using SageMath, $f(X,Y)$ can be factored as the product of five polynomials over quadratic extension $\F_{(q^2)^2}$ of $\F_{q^2}$ as follows
\begin{align*}
 &f(X, Y)=(X + Y)f_1(X, Y)f_2(X, Y)f_3(X, Y)f_4(X, Y), 
\end{align*}
where $f_i(X,Y)=X^2Y^2 + b^{2^{i-1}}(X^2+Y^2) + b^{2^{i+2}} XY + 1 \mbox{~for~} 1 \leq i \leq 4,$  and $b \in \F_{2^4} \subset \F_{q^4}$ is a fifteenth primitive root of unity satisfying $b^4 + b^3 + 1 = 0$. Since $m$ is odd, we have $b \in \F_{q^4} \setminus \F_{q^2}$. 

It is evident that $f(X, Y) = 0$ if and only if either $X=Y$ or $f_i(X,Y)=0$ for any $1 \leq i \leq 4$ $\text{ in } \F_{q^4}$. Let us first consider $f_1(X,Y)=0$. This implies that 
$$\Tr_{2m}^{4m}(f_1(X,Y))=\Tr_{2m}^{4m}(b)(X^2+Y^2)+\Tr_{2m}^{4m}(b^8)XY=0.$$
Since $m$ is odd and $2^{2m}\equiv4\pmod {15}$, we get $\Tr_{2m}^{4m}(b)=b^5$. Substituting the value of $\Tr_{2m}^{4m}(b)$ in the expression for $\Tr_{2m}^{4m}(f_1(X,Y))$, we obtain $X^2+Y^2+b^5(XY)=0$. Raising $X^2+Y^2+b^5(XY)=0$ to the powers $q+1$ and $2$, we get the equations $(X^2+Y^2)^2+b^{5q+5}(XY)^2=0$ and $(X^2+Y^2)^2+b^{10}(XY)^2=0$, respectively. This further renders $b^{5q+5}=b^{10}$, or equivalently, $b^5 \in \F_q$. However, since $b$ is the fifteenth primitive root of unity, the multiplicative order of $b^5$ is three. 
As $m$ is odd, it follows that $b^5 = 1$, which contradicts the fact that $b$ is a primitive fifteenth root of unity. This, in turn, contradicts the assumption that $f_1(X, Y) = 0$ for every $X, Y \in \mu_{q+1}$. A similar approach can be applied to $f_3(X,Y)$ since $\Tr_{2m}^{4m}(f_1(X,Y))=\Tr_{2m}^{4m}(f_3(X,Y))$. Moreover, upon computing $\Tr_{2m}^{4m}(f_2(X,Y))=0$ (or, respectively, $\Tr_{2m}^{4m}(f_4(X,Y))=0$), we obtain $b^5(X^2+Y^2)+XY=0$. Hence, using similar arguments as those applied for $f_1(X,Y)$, we can show that $f_2(X,Y) \ne 0$ and $f_4(X,Y) \ne 0$. Therefore, we conclude that $X=Y$, or equivalently, $G_2(X)$ is injective on $\mu_{q+1}$. 

For the converse part, suppose $m$ is an even positive integer. In this case, we get $2^{m} \equiv 1 \pmod{3}$, and therefore $\gcd(9, 2^m - 1) \geq 3$. Thus, from Lemma~\ref{PP}, we get that $F_2(X)$ does not permute $\F_{q^2}$. 

\end{proof}

In the following theorem, we provide our third class of permutation trinomials over the finite field $\F_{q^2}$.
\begin{thm}\label{T3}
Let $q=2^m$, then the trinomial $F_{3}(X)=X^{7}(X^{7(q-1)} + X^{5(q-1)} + 1)$ permutes $\F_{q^2}$ if and only if $m$ is even and $m \not \equiv 0 \pmod 3$.
\end{thm}
\begin{proof}
From Lemma~\ref{PP}, it is clear that $F_3(X)$ permutes $\F_{q^2}$ if and only if $\gcd(7, q-1) = 1$ and $G_3(X)=X^{7}(X^{7}+X^5+1)^{q-1}$ permutes $\mu_{q+1}$. We first consider the case where $m$ is even and $m \not \equiv 0 \pmod 3$. 
Since $m\not\equiv 0 \pmod{3}$, we obtain that $\gcd(7,q-1)=1$. From Lemma~\ref{roots}, it is known that $X^{7}+X^{5}+1 $ has no roots in $\mu_{q+1}$; hence, $G_3(X)$ can be written as $\dfrac{X^7+X^2+1}{X^7+X^5+1}$ over $\mu_{q+1}$. It is enough to show that if $G_3(X) = G_3(Y)$, for some $X, Y \in \mu_{q+1}$, then $X=Y$. On computing $G_3(X)=G_3(Y)$, we obtain the following equation
\begin{align*}
f(X,Y) & = X^7Y^5+X^5Y^7+X^7Y^2+X^2Y^7+X^2Y^5+X^5Y^2+X^7Y \\&  
\quad + XY^7 + XY^5   +X^5Y + X^7+Y^7+X^2+Y^2+X+Y=0.
\end{align*}
This can be verified using SageMath that $f(X, Y)= (X + Y) \prod_{i=1}^{6} f_i(X, Y)$ over the cubic extension $\F_{(q^2)^3}$ of $\F_{q^2}$, where 
\begin{align*}
		&f_i(X,Y)=Y+b^{21 i}~ \mbox{ for } i\in\{1,2\},\\&
		f_i(X,Y)=X+b^{21(i-2)} ~\mbox{ for }i\in\{3,4\},\\&
		f_5(X,Y)=X^2Y+XY^2+X+Y+1,\\&
		f_6(X,Y)=X^2Y^2+X^2Y+XY^2+X+Y,
\end{align*}
and $b \in \F_{2^6} \subset \F_{q^6}$ is a $63\text{rd}$ primitive root of unity satisfying $b^6+b^4 + b^3 + b+1 = 0$. Since $\gcd(m,3)=1$, we have $b \not \in \F_{q^2}$ but $b \in \F_{q^6}$. 
    
Our goal is now to show that $f_i(X, Y) \ne 0$ for any $1 \leq i \leq 6$ in $\F_{q^6}$. Note that if $f_1(X,Y) = 0$, then $Y^{q+1} = b^{21(q+1)} = 1$. This is possible only when $63 \mid 21(q+1)$, or equivalently when $q \equiv 2 \pmod 3$, which contradicts the assumption that $m$ is even and $m \not \equiv 0 \pmod 3$. Similar arguments can be used to show that $f_i(X,Y) \ne 0,~ 2 \leq i \leq 4$. Next, we assume that $f_5(X,Y)=0$. Then, raising $f_5(X,Y)$ to the powers $q+1$ and $2$, we obtain the equations $(XY+1)^2(X+Y)^2=(XY)^2$ and $(XY+1)^2(X+Y)^2=1$, respectively. This further gives us $XY=1$. Substituting $XY=1$ in the expression of $f_5(X,Y)$, we get $f_5(X,Y)=1$, a contradiction. By applying a similar approach to that used for $f_5(X,Y)$, one can show that $f_6(X, Y) \ne 0$. Thus, $f(X, Y)=0$ holds if and only if $X=Y$.

Conversely, first assume that $F_3(X)$ permutes $\F_{q^2}$ for some $m \equiv 0 \pmod 3$. This gives us $\gcd(7, 2^m-1) \neq 1$, a contradiction to the assumption that $F_3(X)$ is a permutation polynomial over $\F_{q^2}$. Next, suppose $m$ is odd and $m \not \equiv 0 \pmod 3$, then it follows that $q \equiv 2 \pmod 3$. Therefore, we can choose an element $\alpha \in \F_{q^2}^{*}$ having multiplicative order $3(2^m-1)$ and compute $F_3(\alpha)$. Since $\alpha^{3(2^m-1)}=1$, this implies that $ (\alpha^{2^m-1}+1)(1+\alpha^{(2^m-1)}+\alpha^{2(2^m-1)})=0$. Thus, we obtain $F_3(\alpha)=0$, which contradicts the assumption that $F_3(X)$ is a permutation polynomial over $\F_{q^2}$. This completes the proof.
\end{proof}

\section{QM inequivalence with known classes of permutation trinomials}\label{S4}
\begin{table}[hbt]
\caption{Known permutation trinomials over $\mathbb{F}_{2^{2m}}$.}\label{Table1}
\begin{center}
\scalebox{0.7}{
\begin{tabular}{|c|c|c|c|c|c|}
\hline
${i}$&$r$ & $h_i(X)$ & Conditions on $m$ & $f_{i}(X)$ &  Reference\\
\hline
$ 1$&$3$ & $1+X+X^3$ & $m$ is odd & $X^3+X^{q+2}+X^{3q}$ & \cite[Theorem 4.2]{ZLF} \\
\hline
$ 2$&$3$ & $1+X^2+X^3$ & $m$ is odd & $X^3+X^{2q+1}+X^{3q}$ & \cite[Theorem 4.1]{ZLF} \\
\hline
$ 3$&$2$ & $1+X^2+X^3$ & $\gcd(m,3)=1$ & $X^2+X^{2q}+X^{3q-1}$ & \cite[Theorem 3.3]{RG} \\
\hline
$ 4$&$4$ & $1+X+X^3$ & $\gcd(m,3)=1$ & $X^4+X^{q+3}+X^{3q+1}$ & \cite[Theorem 3.1]{RG} \cite[Theorem 2.6]{Li} \\
\hline
$ 5$&$3$ & $1+X^3+X^4$ & $m$ is odd & $X^3+X^{3q}+X^{4q-1}$ & \cite[Theorem 3.5]{RG}  \\
\hline				
$ 6$&$5$ & $1+X+X^4$ & $m$ is odd & $X^5+X^{q+4}+X^{4q+1}$ & \cite[Theorem 3.4]{RG} \cite[Theorem 2.8]{Li} \\
\hline
$ 7$&$5$ & $1+X^3+X^4$ & $m \equiv 2 \pmod 4$ & $X^5+X^{3q+2}+X^{4q+1}$ & \cite[Theorem 3.1]{ZLF} \\
\hline
$ 8$&$4$ & $1+X+X^5$ & $m \equiv 2,4 \pmod 6$ & $X^4+X^{q+3}+X^{5q-1}$ & \cite[Theorem 2.7]{Li}  \\
\hline
$ 9$&$5$ & $1+X+X^5$ & $m \equiv 2 \pmod 4$ & $X^5+X^{q+4}+X^{5q}$& \cite[Theorem 2.4]{Li} \cite[Theorem 4.3]{ZLF} \\
\hline
$10 $&$5$ & $X+X^2+X^5$ & $m \equiv 2 \pmod 4$ &$X^{q+4}+X^{2q+3}+X^{5q}$& \cite[Theorem 3.2]{ZLF} \\
\hline
$11 $&$5$ & $1+X^4+X^5$ & $m \equiv 2 \pmod 4$ & $X^5+X^{4q+1}+X^{5q}$& \cite[Theorem 4.4]{ZLF}  \\
\hline
$12 $&$5$ & $1+X^2+X^6$ & $m \not \equiv 0 \pmod 4, \gcd(m,3)=1$ & $X^5+X^{2q+3}+X^{6q-1}$ & \cite[Theorem 2.10]{Li} \cite[Theorem 3.4]{Harsh} \\
\hline
$13 $&$5$ & $1+X^5+X^6$ & $m \not \equiv 0 \pmod 4, \gcd(m,3)=1$& $X^5+X^{5q}+X^{6q-1}$ & \cite[Theorem 3.2]{Harsh} \\
\hline
$14 $&$7$ & $1+X+X^6$ & $\gcd(m,3)=1$ &$X^7+X^{q+6}+X^{6q+1}$ & \cite[Theorem 3.1]{Harsh}  \\
\hline
$ 15$&$7$ & $1+X^4+X^6$ & $\gcd(m,3)=1$ & $X^7+X^{4q+3}+X^{6q+1}$ & \cite[Theorem 3.3]{Harsh} \\
\hline
$16 $&$7$ & $1+X^5+X^7$ & $m \text{ is even, } m \not \equiv 0 \pmod 3$& $X^7+X^{5q+2}+X^{7q}$ & This paper\\
\hline
$17 $&$9$ & $1+X^6+X^8$ & $m$ is odd &$X^9+X^{6q+3}+X^{8q+1}$ & This paper\\
\hline
$ 18$&$11$ & $1+X^4+X^{10}$ & $m\not\equiv0\pmod5$ & $X^{11}+X^{4q+7}+X^{10q+1}$ & This paper\\
\hline
\end{tabular}
}
\end{center}
\end{table}

A natural question when studying permutation polynomials is whether the proposed classes of permutation polynomials are new or not, and QM equivalence plays an important role in addressing this question. However, determining QM equivalence between two permutation polynomials is not so easy, since there is no general method for this evaluation. We recall the following definition of QM equivalence introduced by Wu, Yuan, Ding and Ma~\cite{wu2017permutation}. 
\begin{defn}{\label{equiv}}
Two permutation polynomials $f(X)$ and $g(X)$ in $\F_{q}[X]$ are called quasi-multiplicative (QM) equivalent if there exists an integer 
$1\leq d < q-1$ with $\gcd(d,q-1)=1$ and $f(X)= \alpha g(\beta X^d)$, where $\alpha,\beta\in\mathbb{F}_{q}^{*}$.
\end{defn}
In this section, we show that the classes of permutation trinomials proposed in this paper are not QM equivalent to each other. Furthermore, we verify that our proposed classes of permutation trinomials are not QM equivalent to the known permutation trinomials mentioned in Table~\ref{Table1}. Now, we recall the strategy introduced in \cite{tu2018class} to show the QM inequivalence between two permutation trinomials, which we also used in this paper.
Let $q=2^m$ and $F(X)=X^r(X^{\alpha(q-1)}+X^{\beta(q-1)}+1)$ be a permutation trinomial over $\F_{q^2}$. Suppose $G(X)=aX^{A}+bX^{B}+cX^{C}$ is any other permutation trinomial over $\F_{q^2}$. Then, to show that $F(X)$ and $G(X)$ are not QM equivalent over $\F_{q^2}$, we proceed by the following two-steps approach. 

\textbf{Step 1.} Show that there does not exist any positive integer $1\leq d < q^2-1$ such that $\gcd(d,q^2-1)=1$ and $\{r,r+\alpha(q-1),r+\beta(q-1)\}_{\pmod{q^2-1}}=\{Ad,Bd,Cd\}_{\pmod{q^2-1}}$, where $A_{\pmod{q^2-1}}=\{a \pmod{q^2-1} \mid a \in A\}$.

\textbf{Step 2.} If the condition in Step 1 does not hold, then compare the coefficients of $F(X)$ with $A_{1}G(A_{2}X^d)$ and show that there does not exist any $A_{1},A_{2}\in\mathbb{F}_{q^2}$ such that $F(X) = A_{1}G(A_{2}X^d)$.

If both steps fail, we conclude that $F(X)$ and $G(X)$ are QM equivalent, denoted as $F(X) \sim G(X)$. For completeness, we present the following lemma to investigate the QM equivalence of the permutation trinomials over $\F_{q^2}$ when $q=2^m$ for $m = 1$ and $2$.

\begin{lem}\label{L1}\cite[Theorem 4.2]{Harsh}
	Let $F(X) = X^r(X^{\alpha(q-1)} + X^{\beta(q-1)}+ 1)$ be a permutation trinomial of $\F_{q^2}$,
	where $\alpha > \beta$ and $r$ are positive integers. Then the following hold
	\begin{enumerate}
		\item $F(X) \sim X$ over $\F_{2^2}$ if and only if $3 \mid \alpha \beta (\alpha-\beta)$.
		\item $F(X) \sim X(X^2+X+1)$ over $\F_{2^2}$ if and only if $3  \nmid \alpha \beta (\alpha-\beta)$.
		\item $F(X) \sim X$ over $\F_{2^4}$ if and only if $5 \mid \alpha \beta (\alpha-\beta)$.
	\end{enumerate}
\end{lem}

\begin{rmk}
Notice that for $m=1$ and $m=2$, $F_1$ and $F_3$ are QM equivalent to $G(X)=X$, using Lemma~\ref{L1}. Similarly, for $m=1$, $F_2$ is QM equivalent to $G(X)=X$.
\end{rmk}

We now discuss the QM equivalence among the permutation trinomials proposed in this paper. First, we see that the classes presented in Theorem~\ref{T1} and Theorem~\ref{T2} are not QM equivalent over $\F_{q^2}$.

\begin{prop}{\label{prop2}}
Let $q=2^m$, where $m >1$ is an odd positive integer and $m \not \equiv 0 \pmod 5$. Then, the permutation trinomials $F_1(X)=X^{11}(X^{10(q-1)}+X^{4(q-1)}+1)$ and $F_2(X)=X^{9}(X^{8(q-1)}+X^{6(q-1)}+1)$ are not QM equivalent over $\F_{q^2}$.
\end{prop}
\begin{proof}
Let $F_1(X)$ and $F_2(X)$ be QM equivalent and suppose there exists an integer 
$1 \leq d \leq q^2-2$ such that $\gcd(d,q^2-1)=1$ with
\[
\{11,\, 4q+7,\, 10q+1 \}_{\pmod{q^2-1}}
= \{ 9d,\, (6q+3)d,\, (8q+1)d\}_{\pmod{q^2-1}}.
\]
Since $3 \mid (q+1)$, we reduce the sets modulo $3$:
\[
\{11,\, 4q+7,\, 10q+1 \}_{\pmod{3}}
= \{ 9d,\, (6q+3)d,\, (8q+1)d\}_{\pmod{3}}.
\]
This yields
\[
\{2,\, q + 1\}_{\pmod{3}}
= \{0,\,(2q + 1)d\}_{\pmod{3}}.
\]
Hence, the only possible case is
\[
2 \equiv (2q + 1)d \pmod{3}
\;\;\Longrightarrow\;\;
2 \equiv 2d \pmod{3}
\;\;\Longrightarrow\;\;
1 \equiv d \pmod{3},
\]
which contradicts
\[
\{2,\, q + 1\}_{\pmod{3}} = \{0,\,(2q + 1)d\}_{\pmod{3}}.
\]

% We consider the following three possible cases.	

% \textbf{Case 1.} Suppose that $11 \equiv 9d\pmod {q^2-1}$. Since $3 \mid (q^2-1)$, it follows that $3 \mid 11$, a contradiction. 

% \textbf{Case 2.} Assume that $11 \equiv (6q+3)d\pmod {q^2-1}$. The argument used in Case 1 also applies here.
	
% \textbf{Case 3.} Let $11 \equiv (8q+1)d\pmod {q^2-1}$. This implies that either $4q+7 \equiv 9d \pmod {q^2-1}$ or $4q+7 \equiv (6q+3)d \pmod {q^2-1}$. First, suppose that $4q+7 \equiv 9d \pmod {q^2-1}$. Then, we obtain 
% $$(8q+1)d = 8(q-1)d+9d = 8(q-1)d+(4q+7) \equiv 11 \pmod {q^2-1},$$ or equivalently, $8(q-1)d+4(q-1) \equiv 0 \pmod {q^2-1}$. This simplifies to $2d+1 \equiv 0 \pmod{q+1}$.  Additionally, from $11 \equiv (8q+1)d = 8(q+1)d-7d\pmod {q^2-1}$, it follows that $7d+11 \equiv 0 \pmod{q+1}$. Now, solving equations $2d+1 \equiv 0 \pmod{q+1}$ and $7d+11 \equiv 0 \pmod{q+1}$, we obtain $5(d+2)  \equiv 0 \pmod{q+1}$. Since $m$ is odd and $m \not \equiv 0 \pmod 5$, we have $\gcd(5,q+1)=1$, implying $d+2 \equiv 0 \pmod{q+1}$. Combining this with $2d+1 \equiv 0 \pmod{q+1}$, we conclude $3 \equiv 0 \pmod {q+1}$, which is impossible.

% Similarly, if $4q+7 \equiv (6q+3)d \pmod {q^2-1}$, an analogous argument leads to the same contradiction.
%     Summarizing the above discussion, we get the desired claim.
\end{proof}
In the following proposition, we show that the classes of permutation trinomials proposed in Theorem~\ref{T1} and Theorem~\ref{T3} are not QM equivalent over $\F_{q^2}$.
\begin{prop}{\label{prop1}}
Let $q=2^m$, where $m$ is an even positive integer such that $m\not\equiv0\pmod{3}$ and $m\not\equiv0\pmod{5}$. Then, the permutation trinomials $F_1(X)=X^{11}(X^{10(q-1)} + X^{4(q-1)} + 1)$ and $F_3(X)=X^{7}( X^{7(q-1)} + X^{5(q-1)} + 1)$ are not QM equivalent over $\F_{q^2}$.
\end{prop}
\begin{proof}
Suppose that the permutation trinomials $F_1(X)$ and $F_3(X)$ are QM equivalent such that there exists a positive integer $1 \leq d < q^2-1$ with $\gcd(d,q^2-1)=1$ and   
\begin{equation}{\label{1eq}}
\{11,4q+7,10q+1\}_{\pmod{q^2-1}}=\{7d,(5q+2)d,7qd\}_{\pmod{q^2-1}}.
\end{equation}
We analyze the following possible cases, all within modulo $(q^2 - 1)$.
    \begin{align*}
        & \text{(1)} \begin{cases}
			11=7d\\
			4q+7=(5q+2)d\\
			10q+1=7qd
		\end{cases} 
        & \text{(2)}  \begin{cases}
			11=7d\\
			4q+7=7qd\\
                10q+1=(5q+2)d
		\end{cases} 
        & \text{(3)}  \begin{cases}
			11=(5q+2)d\\
			4q+7=7d\\
			10q+1=7qd
		\end{cases} 
        \end{align*}
        
\begin{align*}
 & \text{(4)} \begin{cases}
		11=(5q+2)d\\
		4q+7=7qd\\
		10q+1=7d
\end{cases}  
& \text{(5)} \begin{cases}
		11=7qd\\
		4q+7=(5q+2)d\\
		10q+1=7d
\end{cases} 
        &\text{(6)} \begin{cases}
			11=7qd\\
			4q+7=7d\\
			10q+1=(5q+2)d.
		\end{cases}
    \end{align*}
    
\textbf{Case 1.} In this case, we have $10q+1\equiv 7qd\equiv 11q\pmod{q^2-1}$. This implies that $(q-1) \equiv 0 \pmod{q^2-1}$, which is a contradiction.

\textbf{Case 2.} Since $\gcd(3,m)=1$, it follows that $\gcd(7,q^2-1)=1$. Now, from the congruence $4q+7\equiv 11q\pmod{q^2-1}$, we obtain $ 7(q-1)\equiv0\pmod{q^2-1}. $ This contradicts the fact that $\gcd(7,q^2-1)=1$.

\textbf{Case 3.} Using the congruences $4q+7\equiv7d\pmod{q^2-1}$ and $10q+1\equiv7qd\pmod{q^2-1}$, we derive $4q^2-3q-1\equiv3(q-1)\equiv0\pmod{q^2-1}.$
This yields a contradiction, since $m$ is an even positive integer and $\gcd(3, q+1)  = 1$.  

Similarly, for the Case 4, Case 5 and Case 6, we arrive at the congruences $3(q-1)\equiv 0 \pmod{q^2-1}$, $q-1\equiv 0 \pmod{q^2-1}$ and $7(q-1)\equiv 0 \pmod{q^2-1}$, respectively. Therefore, we conclude that no positive integer $d$ satisfies Equation~\eqref{1eq} with $\gcd(d,q^2-1)=1$. This completes the proof.
\end{proof}
Next, we recall the following result, which gives the QM equivalence among the permutation trinomials listed in Table~\ref{Table1}.
\begin{lem}{\label{equi}}\cite[Theorem 4.1]{Harsh}
Let $f_i(X)=X^{r_i} h_i\left(X^{q-1}\right)$ for $1 \leq i \leq 12$, as mentioned in Table~\ref{Table1}. Then the following hold over $\F_{q^2}$, where $q=2^m$.
\begin{enumerate}
\item $f_1(X)=X^3\left(X^{3(q-1)}+X^{q-1}+1\right) \sim f_2(x)=X^3\left(X^{3(q-1)}+X^{2(q-1)}+1\right)$.
\item $f_3(X)=X^2\left(X^{3(q-1)}+X^{2(q-1)}+1\right) \sim f_4(X)=X^4\left(X^{3(q-1)}+X^{q-1}+1\right)$.
\item $f_5(X)=X^3\left(X^{4(q-1)}+X^{3(q-1)}+1\right) \sim f_6(X)=X^5\left(X^{4(q-1)}+X^{q-1}+1\right)$.
\item $f_7(X)=X^5\left(X^{4(q-1)}+X^{3(q-1)}+1\right) \sim f_{10}(X)=X^5\left(X^{5(q-1)}+X^{2(q-1)}+X^{q-1}\right)$.
\item $f_9(X)=X^5\left(X^{5(q-1)}+X^{q-1}+1\right) \sim f_{11}(X)=X^5\left(X^{5(q-1)}+X^{4(q-1)}+1\right)$.
\end{enumerate}
\end{lem}

Next, we verify that our first class of permutation trinomials, $F_{1}(X) = X^{11}(X^{4(q-1)}+X^{10(q-1)}+1)$, is not QM equivalent to all previously known permutation trinomials $f_i$ for each $1\leq i \leq 15$ listed in Table~\ref{Table1}. The QM equivalence between the other permutation trinomials identified in this paper and the known ones listed in Table~\ref{Table1} can be investigated in a similar manner. 

\begin{thm}
Let $F_1(X)=X^{11}(X^{10(q-1)} + X^{4(q-1)} + 1)$. Then, $F_1$ is QM inequivalent to all the permutation trinomials $f_i$ listed in Table~\ref{Table1} for $1 \leq i \leq 15$ over $\F_{q^2}$.
\end{thm}
\begin{proof}
% First, observe that Step 1 is enough to determine the QM inequivalence of the permutation trinomial $F_1$ with $f_i$ for each $i=1,\ldots, 15$ over $\F_{q^2}$. 
% Let $f_i(X)$ be expressed as $X^a+X^b+X^c$.  Then, it is sufficient to show that there does not exist any integer $1\leq d\leq q^2-2$ with $\gcd(d,q^2-1)=1$, satisfying 
%     \begin{equation}{\label{1}}
%         \{11,4q+7,10q+1\}_{\pmod{q^2-1}}=\{ad,bd,cd\}_{\pmod{q^2-1}},
%     \end{equation} or 
%     \begin{equation}
%         \{11d,(4q+7)d,(10q+1)d\}_{\pmod{q^2-1}}=\{a,b,c\}_{\pmod{q^2-1}}.
%     \end{equation}
From Lemma~\ref{equi}, it is sufficient to verify the QM inequivalence of $F_1$ with $f_i$ for $i \in \{1,3,5,7,8,9,12,13,14,15\}$.  Suppose, for contradiction, $F_{1}\sim f_{1}$. Then, it is sufficient to show that there does not exist any integer $1\leq d\leq q^2-2$ with $\gcd(d,q^2-1)=1$, satisfying 
$$\{11d,(4q+7)d,(10q+1)d\}_{\pmod{q^2-1}}=\{3,3q,q+2\}_{\pmod{q^2-1}}.$$ 
Clearly, neither $11d  \equiv 3 \pmod{q^2-1}$, nor $11d  \equiv 3q \pmod{q^2-1}$ can hold, as $3 \nmid 11$. Therefore, the only remaining possibility is $11d \equiv q+2                                                                                                                                                                                                                                                                                                                                                                                                                                                                                                                                                                                                                                                                                                                                                                                                                                                                                                                                                                                                                                                                                                                                                                                                                                                                                                                                                                                                                                                                                                                                                                                                                                                                                                                                                                                                                                                                                                                                                                                                                                                                                                                                                                                                                                                                                                                                                                                                                                                                                                                                                                                                                                                                                                                                                                                                                                                                                                                                                                                                                                                                                                                                                                                                                                                                                                                                                                                                                                                                                                                                                                                                                                                                                                                                                                                                                                                                                                                                                                                                                                                                                                                                                                                                                                                                                                                                                                                                                                                                                                                                                                     \pmod{q^2-1}$. Further, assuming $(10q+1)d \equiv 3 \pmod{q^2-1}$ and $(4q+7)d \equiv 3q \pmod{q^2-1}$, solving these congruences simultaneously leads to $q+1\mid 3$, which is a contradiction. Similarly, a contradiction arises in the case where $(10q+1)d \equiv 3q \pmod{q^2-1}$ and $(4q+7)d \equiv 3 \pmod{q^2-1}$. Hence, $F_{1}$ is not QM equivalent to $f_{1}$. 

Using similar arguments, we conclude that $F_1$ is QM inequivalent to $f_i$ for each $i\in\{1,3,5,9,13\}$, as in each case, one of the exponent is a multiple of $q$ times another exponent. Next, suppose that $F_{1}\sim f_{7}.$  This renders the following equivalence
$$\{11d,(4q+7)d,(10q+1)d\}_{\pmod{q^2-1}}=\{5,3q+2,4q+1\}_{\pmod{q^2-1}},$$
where $1\leq d\leq q^2-2$ is an integer with $\gcd(d,q^2-1)=1$. 
Recall that $F_{1}$ and $f_{7}$ are permutation trinomials over $\mathbb{F}_{q^2}$ if and only if $m\equiv 2\pmod4$ and $m\not\equiv0\pmod5$, which implies that $5\mid{q^2-1}$. Now, considering the three possible cases: $11d\equiv5\pmod{q^2-1}$, or $(4q+7)d\equiv5\pmod{q^2-1}$, or $(10q+1)d\equiv5\pmod{q^2-1}$, we obtain the congruences $11\equiv0\pmod{5}$, or $2(2q+1)\equiv 0\pmod5$, or $1\equiv 0\pmod 5$, respectively. Since these congruences are not possible, it follows that $F_{1}\not\sim f_{7}$. By a similar argument, we also conclude that $F_{1}\not\sim f_{12}$.

Now, we analyze the QM equivalence of $F_{1}$ with $f_{i}$ for each $i\in\{8,14,15\}$. First, assume that $F_{1}\sim {f_{8}}$. Then for an integer $1\leq d\leq q^2-2$ with $\gcd(d,q^2-1)=1$, we obtain,
$$\{11d,(4q+7)d,(10q+1)d\}_{\pmod{q^2-1}}=\{4,q+3,5q-1\}_{\pmod{q^2-1}}.$$ 
Next, we consider six cases possible from the above equality and arrive at the congruence relation $A(q-1)\equiv 0\pmod{q^2-1}$, where $A\in\{10,46,44,20,14,34\}$. Since these congruences are not possible, it follows that $F_{1}\not\sim f_{8}$. Following a similar approach, we can shown that $F_{1}$ is not QM equivalent to $f_{14}$ and $f_{15}$.  
\end{proof}

\section{Nonexistence of a class of permutation trinomials}{\label{S5}}
In this section, we give a class of trinomials that does not permute $\F_{q^2}$. Let $F(X,Y)$ define an algebraic curve over $\F_q$. A rational point $(X, Y) \in \F_q \times \F_q$ is a point on this curve such that $F(X,Y)=0$. The collection of all rational points of $F(X, Y)$ is given by the following set 
$$V_{\F_{q}^2}(F) = \{(X,Y) \in \F_{q}^2 \mid F(X,Y) = 0\}.$$ We denote by $\mathbb{P}_2(\F_{q})$, the projective plane over $\F_q$. For a homogeneous polynomial
$f(X,Y,Z) \in \F_q[X,Y,Z]$, the set of zeros of $f$ in $\mathbb{P}_2(\F_{q})$ is defined as
$$V_{\mathbb{P}_2(\F_{q})}(f) = \{(X,Y,Z) \in \mathbb{P}_2(\F_{q}) \mid f(X,Y,Z) = 0\}. $$
In addition, a polynomial $F(X,Y) \in \F_q[X,Y]$ is said to be absolutely irreducible if it is irreducible in $\overline \F_q[X,Y]$, where $\overline \F_q$ denotes the algebraic closure of $\F_q$.
We use the following version of the Hasse-Weil theorem by Aubry and Perret in our result.
\begin{lem} (Aubry-Perret bound) \cite[Corollary 2.5]{Aubry} \label{HW}
Let $\mathcal{C} \subset \mathbb{P}_2(\F_{q})$ be an absolutely irreducible curve of degree $d$. Then
$$|\#V_{\mathbb{P}_2(\F_{q})}(\mathcal{C})-(q+1)| \leq (d-1)(d-2)q^{1/2},$$
where $| \cdot |$ denotes the usual absolute value.
\end{lem}

In the following lemma, we show the absolute irreducibility of a curve over $\F_q$, which, together with Lemma~\ref{HW}, give us a class of trinomials that does not permute $\F_{q^2}.$
\begin{lem}\label{ab}
Let $H(X,Y)\in \F_{q}[X,Y]$ be the polynomial $X^{16}+AX^8+(B+1)X^4+AX^2+BX+A$, where $A=Y^8+Y^4+Y^3+Y^2+Y$ and $B=Y^8+Y^7+Y^4+Y^2$. Then, $H(X, Y)$ is absolutely irreducible over $\F_q$.
\end{lem}
\begin{proof}
Notice that $H(X,Y)$ can be considered as a polynomial in variable $X$ over $\overline \F_q[Y]$. Moreover, it is primitive as the gcd of the coefficients is $1$. Hence, it is sufficient to show that $H(X,Y)$ is irreducible over $\overline \F_q(Y)$ or equivalently, $\left[\overline\F_q(X,Y):\overline\F_q(Y )\right] = 16$.

Let $Z=X^4+X$, then we can write $H(X,Y)$ as $Z^4+AZ^2+BZ+A=0$, a polynomial in $Z$ over $\overline \F_q[Y]$. Therefore, in order to show  $\left[\overline\F_q(X,Y):\overline\F_q(Y )\right] = 16$, we need to prove that,  $\left[\overline\F_q(Z, Y ) :\overline \F_q(Y)\right] = 4$ and $\left[\overline\F_q(X,Y) : \overline\F_q(Z,Y)\right] = 4$. Furthermore, suppose that $Z=X^4+X=X^4+X^2+X^2+X$ and $Z_1=X^2+X$. 
	
\begin{center}
		\begin{tikzpicture}
			\node (Q1) at (0,0) {$\overline \F_q(Y)$};
			\node (Q2) at (0,2) {$\overline \F_q(Z,Y)$};
			\node (Q3) at (0,4) {$\overline \F_q(Z_1,Y)$};
			\node (Q4) at (0,6) {$\overline \F_q(X,Y)$};
			% \draw (Q1)--(Q3);
			\draw[-] (Q1) -- (Q2) node[midway, right] {$4$};
			\draw[-] (Q2) -- (Q3) node[midway, right] {$2$};
			\draw[-] (Q3) -- (Q4) node[midway, right] {$2$};
		\end{tikzpicture}   
	\end{center}
Thus, it is sufficient to show the following three claims.
\begin{enumerate}
\item $\left[\overline\F_q(Z, Y ) :\overline \F_q(Y)\right] = 4$, that is, $Z^4+AZ^2+BZ+A$ is irreducible over $\overline \F_q(Y)$.
\item $\left[\overline\F_q(Z_1, Y ) :\overline \F_q(Z,Y)\right] = 2$, that is, $Z_{1}^2+Z_1+Z$ is irreducible over $\overline \F_q(Z,Y)$.
\item $\left[\overline\F_q(X, Y ) :\overline \F_q(Z_1,Y)\right] = 2$, that is, $X^2+X+Z_1$ is irreducible over $\overline \F_q(Z_1,Y)$.
\end{enumerate}
	
{\bf Claim 1.}  Since gcd of coefficients of $Z^4+AZ^2+BZ+A$ is one, it is primitive over $\overline \F_q [Y]$, and hence it is enough to show that it is irreducible over $\overline \F_q [Y]$. Assume that $Z^4+AZ^2+BZ+A$ factors as the product of two quadratics over $\overline \F_q[Y]$, say
	\begin{equation}\label{neq}
		Z^4+AZ^2+BZ+A=(Z^2 + rZ + s)(Z^2 + rZ + t).
	\end{equation}
Equating coefficients in the above expression, we get $A=r^2+s+t=st$ and $B=r(s+t).$ Therefore, $B = r(s + t) = r(A + r^2)$ and comparing degree in terms of $Y$ on the both side, we get $8 = \max\{\text{degree}(r) + 8, 3 \text{degree}(r) \}$, we obtain a contradiction.

% Moreover, from the expressions of $A$ and $B$, it follows that $r$ must satisfy the equation $r^3+Ar+B=0$, which can be further simplified to 
% $$(r+Y^4+Y)(r^2+(Y^4+Y)r+Y^4+Y^3+Y)=0.$$ 
% Let $\alpha$ and $\beta$ be roots of $r^2+(Y^4+Y)r+Y^4+Y^3+Y$ over $\overline\F_q[Y]$, then $Y$ divides $\alpha\beta=Y(Y^3+Y^2+1)$. This implies that $Y$ divides either $\alpha$ or $\beta$, but then $\alpha+\beta$ will contain a constant term $1$ in its expression, which is not true. Hence, $r^2+(Y^4+Y)r+Y^4+Y^3+Y$ is irreducible over $\overline\F_q[Y]$. Therefore, the only possible value of $r$ is $Y^4+Y$. Substituting $r=Y^4+Y$ into the expression of $A$ and $B$, we obtain $s+t=Y^4+Y^3+Y$ and $st=Y^8+Y^4+Y^3+Y^2+Y$. More precisely, we can say that $s$ and $t$ satisfy the equation $x^2+(s+t)x+st=0$ over $\overline \F_q(Y)$.  Then, using an argument similar to that in the previous case for the quadratic polynomial in $r$, we see that $x^2+(s+t)x+st$ is irreducible over $\overline \F_q[Y]$.
	
As a consequence, we conclude that $Z^4+AZ^2+BZ+A$ can not be expressed as the product of two quadratic polynomials over $\overline \F_q[Y]$. Similarly, it can be shown that $Z^4+AZ^2+BZ+A$ can not be factored into a product of a degree-three irreducible polynomial and a linear factor over $\overline \F_q[Y]$.
	
{\bf Claim 2.}  Using the fact that the gcd of coefficients of $Z_{1}^2+Z_1+Z \in \overline \F_q [Z,Y]$ is one ($\overline \F_q [Z,Y]$ is an integral domain and indeed a UFD), it is primitive over $\overline \F_q [Z,Y]$. Hence, it is sufficient to show that it is irreducible over $\overline \F_q [Z,Y]$.  It is straightforward to verify that $Z_{1}^2+Z_1+Z$ is irreducible over $\overline \F_q [Z,Y]$.
	
{\bf Claim 3.}  It follows along the same lines as Claim 2.
	
Therefore, we conclude that $H(X,Y)$ is absolutely irreducible over $\F_q$.	
\end{proof}

Next, we present a class of trinomials 
$F(X)=X^{9}( X^{7(q-1)} + X^{3(q-1)} + 1)$ that does not permute $\F_{2^{2m}}$ for $m>3$.

\begin{thm}
Let $q=2^m$. Then, the trinomial $F(X)=X^{9}( X^{7(q-1)} +X^{3(q-1)} + 1)$ is not a permutation trinomial over $\F_{q^2}$ for $m>3$.
\end{thm}

\begin{proof}
It is known from Lemma~\ref{PP} that $F(X)$ is a permutation polynomial over $\F_{q^2}$ if and only if $\gcd(9,q-1)=1$ and $G(X)=X^9(X^7+X^3+1)^{q-1}$ permutes $\mu_{q+1}$. Moreover, from Lemma~\ref{roots}, $G(X)$ can be written as $G(X)=\dfrac{X^{9}+X^{6}+X^{2}}{X^7+X^3+1}$ over $\mu_{q+1}$. Thus, $F(X)$ is a permutation polynomial over $\F_{q^2}$ if and only if $m$ is odd and $G(X)=\dfrac{X^{9}+X^{6}+X^{2}}{X^7+X^3+1}$ permutes $\mu_{q+1}$. 

To prove that $F(X)$ is not a permutation polynomial over $\F_{q^2}$, we show that $G(X)$ does not permute $\mu_{q+1}$. Define $\phi(X) :=\dfrac{X+\omega^2}{X+\omega}$, a bijection from $\F_q$ to $\mu_{q+1}\setminus \{1\}$, where $\omega \in \F_{q^2} \setminus \F_q$ and $\omega^2+\omega+1=0$. Clearly, $G$ permutes $\mu_{q+1}$ if and only if $G \circ \phi$ is a bijection from $\F_q$ to $\mu_{q+1}\setminus \{1\}$. Alternatively, $G$ permutes $\mu_{q+1}$ if and only if $\phi^{-1} \circ G \circ \phi$ is a bijection from $\F_q$ to $\F_q$. Therefore, to show that $G$ is not a bijection over $\mu_{q+1}$, we prove that $\phi^{-1} \circ G \circ \phi$ is not a bijection on $\F_q$. The map $G \circ \phi : \F_q \rightarrow \mu_{q+1}$ is given by
$$
G (\phi(X))= \dfrac{X^{9}+\omega^2 X^8+ X^5+\omega X^4+ X^3+\omega X^2+\omega^2 X+\omega^2}{X^9+\omega X^8+X^5+\omega^2X^4+X^3+\omega^2X^2+\omega X +\omega}.
$$
Since $\phi^{-1}(X)=\dfrac{\omega X+\omega^2}{X+1}$, we get that 
$
\phi^{-1}(G (\phi(X)))= \dfrac{X^9+X^5+X^4+X^3+X^2}{X^8+X^4+X^2+X+1}.
$
It is easy to see that the numerator of $\phi^{-1}(G (\phi(X))= X^2(X^7+X^3+X^2+X+1)$ vanishes on $\F_q$ when either $X=0$ or $X^7+X^3+X^2+X+1$ splits over $\F_q$. Since $X^7+X^3+X^2+X+1$ is an irreducible polynomial over $\F_2$, it splits in $\F_{2^{7t}}$ for a positive integer $t$. Therefore, $\phi^{-1} \circ G \circ \phi$ does not permute $\F_q$ when $q=2^{7t}$, where $t$ is a positive integer.

Now, to prove that $\phi^{-1} \circ G \circ \phi$ is not a permutation on $\F_q$ when $m\not\equiv 0 \pmod{7}$, we show that there exist  $X \in  \F_{q}$ and $Y \in \F_{q}^{*}$ such that $\phi^{-1}\left(G (\phi(X+Y))\right)  +\phi^{-1}\left(G (\phi(X))\right)=0$. The expression for $\phi^{-1}(G (\phi(X+Y)))+\phi^{-1}(G (\phi(X)))$ is given by $\dfrac{YH(X,Y)}{D(X,Y)}$, where
\begin{align*}
 H(X,Y)=~ & X^{16} + X^8Y^8 + X^8Y^4 + X^4Y^8 + X^8Y^3 + X^4Y^7 + X^8Y^2 + X^2Y^8 + X^8 Y \\&
 \quad + X Y^8+X^4 Y^4 + XY^7  +Y^8+ X^4 Y^2 + X^2 Y^4 + X^2 Y^3 + XY^4 + X^4 \\&
 \quad + X^2Y^2 + Y^4 + X^2Y + XY^2 + Y^3 + Y^2 + Y, \\
D(X,Y) =~ & (X^8+X^4+X^2+X+1)(X^8+Y^8+X^4+Y^4+X^2+Y^2+X+Y+1).
\end{align*}
We can simplify $H(X,Y)$ as $X^{16}+AX^8+(B+1)X^4+AX^2+BX+A \in \F_{q}[X,Y]$, where $A=Y^8+Y^4+Y^3+Y^2+Y$ and $B=Y^8+Y^7+Y^4+Y^2$. Note that $\phi^{-1} \circ G \circ \phi$ is not a permutation over $\F_q$ if and only if $H(X,Y)$ has a rational point $(X,Y) \in \F_{q} \times \F_{q}$ with $Y \neq 0$. Since $H(X,Y)$ is absolutely irreducible by Lemma~\ref{ab}, we can use Lemma~\ref{HW} to compute a bound on the number of rational points for $H(X,Y)$. Let ${V_{\F_{q^2}}(H) = \{ (X, Y) \in \F_{q}^2 : H(X,Y)=0\}}$ be the set of all rational points on $H(X,Y)$. We observe that the degree of $H(X,Y)$ is $16$ and the number of zeros of $H(X,Y)$ at infinity are two. This follows because if $h(X,Y,Z)$ is homogenized form of $H(X,Y)$ then setting $Z=0$ in the homogenized polynomial, we get exactly two solutions $(0:1:0)$ and $(1:1:0)$ of  $h(X,Y,0)=0$ in the projective space $\mathbb{P}_2(\F_{q})$. Thus, by applying Lemma~\ref{HW}, we obtain the bound on number of rational points of $H(X,Y)$ as follows
\begin{align*}
|\#V_{\F_{q^2}}(H)| & \geq 2^m+1-2^{8+\frac{m}{2}}-2\geq  2^{8+\frac{m}{2}}(2^{\frac{m}{2}-8}-1)-1.
\end{align*}
Clearly when $m \geq 18$, we have $2^{\frac{m}{2}-8}-1 \geq 1$. Therefore, for $m \geq 18$ 
$$|\#V_{\F_{q^2}}(H)| \geq  2^{8+\frac{m}{2}}-1 \geq 2^{17}-1.$$
Additionally, the number of zeros of $H(X,0)=X^{16}+X^4$ in $\F_q$ is $2$ which is strictly less than $|\#V_{\F_{q^2}}(H)|$
for $m \geq 18$. Hence, $H(X,Y)$ must have a zero $(X,Y) \in \F_{q}\times \F_{q}$ with $Y \neq 0$, implying that $F(X)$ is not a permutation polynomial over $\F_{q^2}$ for $m \geq 18$. Furthermore, using SageMath, for $m \not \equiv 0 \pmod 7$ and $m$ odd, we confirm that this result holds for smaller values of $m$, such as $m=5,9,11,13,15,17$. This completes the proof.
\end{proof}

\section{Resolution of a conjecture}{\label{S6}}
In this section, we provide a proof of the conjecture proposed in~\cite{Harsh}. As part of the proof, we begin with the following lemma, which is used in our argument.

\begin{lem}\label{lc1}
Let $q=p^m$ and $F(X)=X^r(X^{\alpha(q-1)}+X^{\beta(q-1)} + 1)$, where $m, \alpha, \beta $ and $r$ are positive integers. If $\gcd(\alpha,\beta,r,q+1) \neq 1$, then $F(X)$ is not a permutation polynomial over $\F_{q^2}.$ 
\end{lem}
\begin{proof}
Suppose that $F(X)$ is a permutation polynomial over $\mathbb{F}_{q^2}$ and $\gcd(\alpha,\beta,r,q+1) \neq 1$. Let $\mathbb{F}_{q^2}^{*}=<\omega>$ and define $Y=\omega^{\frac{q^2-1}{\gamma}}$, where $\gamma=\gcd(\alpha,\beta,r,q+1)$. Then, we have
$$F(Y)=\omega^{\frac{r}{\gamma}(q^2-1)}\Big(\omega^{\frac{\alpha(q-1)}{\gamma}(q^2-1)}+\omega^{\frac{\beta(q-1)}{\gamma}(q^2-1)}+ 1\Big )=F(1).$$
This contradicts the assumption that $F(X)$ is a permutation over $\F_{q^2}$.
\end{proof}

\begin{thm}{\label{conjucture}}\cite[Conjecture 1]{Harsh}
Let $r, m$ and $\alpha>\beta$ be positive integers and $q$ be any prime power. If $F(X)=X^r(X^{\alpha(q-1)} + X^{\beta(q-1)} + 1)$ and $G(X)=X^{2\alpha-r}(X^{\alpha(q-1)} + X^{(\alpha-\beta)(q-1)} + 1)$ permutes $\F_{q^2}$, then they are QM equivalent.
\end{thm}
\begin{proof}
We know that the permutation trinomials $F(X)$ and $G(X)$ are QM equivalent over $\F_{q^2}$ if the sets of exponents of $F(X)$ and $G(X)$, namely $A$ and $B$, satisfy the following relation
\begin{align*}
    A_{\text{mod~}(q^2-1)} & = \{rd+\alpha(q-1)d,rd+\beta(q-1)d,rd\}_{\text{mod~}(q^2-1)}\\
    & = \{2\alpha-r+\alpha(q-1),2\alpha-r+(\alpha-\beta)(q-1),2\alpha-r\}=B_{\text{mod~}(q^2-1)}.
\end{align*}
Obviously, on comparing the sets $A_{\text{mod~}(q^2-1)}$ and $B_{\text{mod~}(q^2-1)}$, we have six possible cases. However, only one among the six cases hold true, stated as follows
\begin{equation}\label{c5}
\begin{cases}
& rd+\alpha(q-1)d  \equiv 2\alpha-r \pmod {q^2-1},\\
& rd+\beta(q-1)d  \equiv \alpha-r+(\alpha-\beta)(q-1) \pmod {q^2-1},\\
& rd  \equiv 2\alpha-r+\alpha(q-1) \pmod {q^2-1}.
\end{cases}
\end{equation}	
This is because we can discard the remaining five possible cases, using the argument that $q+1$ does not divide any of the elements in the set $\{2\alpha\beta-\alpha r, \alpha^2+\beta^2-\alpha\beta, \alpha^2-\beta^2, \beta^2-2\alpha\beta \}$ for large $q$.

We now split our analysis into two cases depending on whether $q$ is even or odd. First, suppose that $q$ is even. We simplify System~\eqref{c5} to obtain the following system
\begin{align}\label{c1}
\begin{cases}
(d+1)\eta & \equiv 0 \pmod {q+1}, ~\mbox{where } \eta \in \{\alpha, \beta, r\}, \\
(d+1)r & \equiv \alpha (q+1) \pmod {q-1} .
\end{cases}
\end{align}
Next, we consider the following two cases

\textbf{Case 1.} Let $\gcd(\alpha,q+1)=1$, or $\gcd(\beta,q+1)=1$, or $\gcd(r,q+1)=1$. WLOG assume $\gcd(\alpha,q+1)=1$. Then, from System~\eqref{c1}, we can write
	\begin{equation*}
		\begin{cases}
			d & \equiv -1 \pmod {q+1}, \\
			d & \equiv (\alpha (q+1))r^{-1}-1 \pmod {q-1} .
		\end{cases}
	\end{equation*}
 We apply Chinese Remainder Theorem~\ref{CRT} on the above system, as $\gcd(q-1,q+1)=1$. Let $n_1=q+1, n_2=q-1$ and $N=q^2-1$. Then, we have $N_1=\frac{N}{n_1}=q-1$ and $N_2=\frac{N}{n_2}=q+1$. Moreover, it is easy to see that the multiplicative inverse of $N_1$ modulo $n_1$ and the multiplicative inverse of $N_2$ modulo $n_2$ are both $2^{m-1}$. Therefore, applying the proof of CRT, we obtain
$$d \equiv -q^2+2^{m-1}\alpha r^{-1}(q+1)^2 \pmod{q^2-1}.$$
Similarly, when $\gcd(\beta,q+1)=1$ or $\gcd(r,q+1)=1$, we reach at the same value of $d$.

\textbf{Case 2.} Let $\gcd(\alpha,q+1)=\delta_1\neq1$, $\gcd(\beta,q+1)=\delta_2\neq1$ and $\gcd(r,q+1)=\delta_3\neq1$, then we have,
\begin{align*}\label{c6}
		\begin{cases}
			d & \equiv -1 \pmod {(q+1)/\delta_i} \text{ for } i \in \{1,2,3\},\\
			d & \equiv \alpha (q+1)r^{-1}-1 \pmod {q-1} .
		\end{cases}
	\end{align*}
Since, $\gcd(q-1,(q+1)/\delta_i)=1$ for $1 \leq i \leq 3$, we apply the CRT to the above system independently for each $i \in \{1,2,3\}$, yielding
\begin{equation}\label{c7}
    d \equiv -q^2+2^{m-1}\alpha r^{-1}(q+1)^2 \pmod{(q^2-1)/\delta_i}, \text{ for } i \in \{1,2,3\}.
\end{equation}
We further consider the following two subcases depending upon $\gcd(\delta_1,\delta_2,\delta_3)$.

\textbf{Subcase 2(a).} Suppose that $\gcd(\delta_1,\delta_2,\delta_3) = 1$. From Equation~\eqref{c7}, we get
$d \equiv -q^2+2^{m-1}\alpha r^{-1}(q+1)^2 \pmod{\ell}$, $\ell=\lcm((q^2-1)/\delta_1, (q^2-1)/\delta_2, (q^2-1)/\delta_3)$. Through some elementary calculations, we obtain $\ell=q^2-1$. Thus, we can express ${d \equiv -q^2+2^{m-1}\alpha r^{-1}(q+1)^2 \pmod{q^2-1}}$.

\textbf{Subcase 2(b).} Let $\gcd(\delta_1,\delta_2, \delta_3) \neq 1$. From Lemma~\ref{lc1}, it follows that $F(X)$ is not a permutation polynomial over $\F_{q^2}$, a contradiction. 

Summarizing the above discussion, we get $d \equiv -q^2+2^{m-1}\alpha r^{-1}(q+1)^2 \pmod{q^2-1}$ or more precisely, $d=\alpha(q+1)r^{-1}-1 $. Moreover, note that $\gcd(d,q^2-1)=1$. This completes the proof for $q$ as an even prime power. 

Next, we consider the scenario where $q$ is an odd prime power.  Let $n$ be a positive integer such that $2^n \mid (q+1)$ and $2^{n+1} \nmid (q+1)$. Given that $\gcd(r,q-1)=1$, there exists a positive integer $s$ such that $rs\equiv 1\pmod {2^n(q-1)}$. We deduce the following equations from System~\eqref{c5}
\begin{align}{\label{c1s1}}
        \begin{cases}
        (d+1)\eta&  \equiv 0 \pmod {q+1}, ~\mbox{where } \eta \in \{\alpha, \beta, r\}, \\
        d & \equiv \alpha (q+1)s-1 \pmod {2^n(q-1)}.
        \end{cases}
\end{align}

Now, we analyze the following two possible cases.
			  
\textbf{Case 1.} Let $\gcd(\alpha,q+1)=1$, or $\gcd(\beta,q+1)=1$, or $\gcd(r,q+1)=1$. WLOG assume $\gcd(\alpha,q+1)=1$. Then, from System~\eqref{c1s1}, we have
\begin{equation*}
\begin{cases}
d & \equiv -1 \pmod {(q+1)/2}, \\
d & \equiv \alpha (q+1)s-1 \pmod {(q-1)/2}.
\end{cases}
\end{equation*}
Since $\gcd((q+1)/2,(q-1)/2)=1$, we apply CRT to the above system of equations and obtain that $d\equiv \alpha (q+1)s-1 \pmod {(q^2-1)/4}$.
		
\textbf{Case 2.}  Suppose that $ \lambda_{1}=\gcd(\alpha,q+1) \ne 1 $, $ \lambda_{2}=\gcd(\beta,q+1) \ne 1$ and $ \lambda_{3}=\gcd(r,q+1) \ne 1$. 
Now, we need to examine the following two subcases

\textbf{Subcase 2(a).} Assume that  $\gcd(\lambda_{1},\lambda_{2},\lambda_{3})=1$. Then, one can write $\lambda_{i}=2^k\delta_{i}$ with $k\in \{0,1\}$, depending on whether $\lambda$ is odd or even, respectively. This leads to the following system
		\begin{equation*}
		\begin{cases}
		d & \equiv -1 \pmod {(q+1)/2 \delta_{i}}
        \text{ for } i \in \{1,2,3\},  \\
		d & \equiv \alpha (q+1)s-1 \pmod {(q-1)/2}.
	\end{cases}
\end{equation*}
By applying CRT to the system of equations provided above, we can derive 
$ d\equiv \alpha{(q+1)}s-1 \pmod{(q^2-1)/4\delta_{i}}, \text{ where } i \in \{1,2,3\}.$ This further results in $ d\equiv \alpha(q+1)s-1 \pmod{\ell},$ where $\ell$ denotes $\lcm((q^2-1)/4\delta_{1},(q^2-1)/4\delta_{2},(q^2-1)/4\delta_{3})$.
In the similar manner as the case of even characteristic, one can compute $\ell$ and show that
\begin{equation}\label{c2imp}
d\equiv \alpha(q+1)s-1 \pmod{(q^2-1)/4}.
\end{equation}

\textbf{Subcase 2(b).} Let $\gcd(\lambda_{1},\lambda_{2}, \lambda_{3}) \neq 1$. Then, from Lemma~\ref{lc1}, we observe that $F(X)$ is not a permutation polynomial over $\F_{q^2}$, a contradiction.
	
Notice that $d\equiv \alpha(q+1)s-1 \pmod{(q^2-1)/4}$ holds in Case 1 and Case 2. From last equation of the System~\eqref{c1s1} and Equation~\eqref{c2imp}, we have the following system of equations
\begin{equation*}
\begin{cases}
d & \equiv \alpha (q+1)s-1 \pmod {(q^2-1)/4}, \\
d & \equiv \alpha (q+1)s-1 \pmod {2^n(q-1)}.
\end{cases}
\end{equation*}
One can easily see that $\lcm((q^2-1)/4,2^n(q-1))=q^2-1.$ Therefore, we have
\begin{equation*}
d\equiv \alpha(q+1)s-1 \pmod{q^2-1}.
\end{equation*} 
This completes the proof.
\end{proof}

\begin{rmk}
After independently proving Conjecture~\cite[Conjecture 1]{Harsh} in Theorem~\ref{conjucture}, we discovered that it had already been established as a special case of a result by Yadav, Singh, and Gupta~\cite[Theorem 3.3]{Yadav}. However, our proof differs significantly from theirs, particularly in that we explicitly compute the exponent $d$ rather than assuming it, as was done in their work. Thus, we believe it is still worthwhile to include this proof in our paper.
\end{rmk}

\section{Conclusion}\label{S7}
In this paper, we construct three new classes of permutation trinomials for $(\alpha,\beta, r) \in \{ (7,5,7), (8,6,9), (10,4,11) \}$ over the finite field $\F_{2^{2m}}$.  Moreover,  we analyze the QM equivalence among the newly determined permutation trinomials and the previously known ones. Finally, using some algebraic geometric techniques, we show that a specific class of trinomials does not permute $\F_{2^{2m}}$.


\begin{thebibliography}{99}

\bibitem{AW2007}A. Akbary, Q. Wang, {\it On polynomials of the form $X^r f(X^{(q-1)/l)}$}, Int. J. Math. Math. Sci. (2007) 023408.

\bibitem{Aubry} Y. Aubry, M. Perret, {\it A Weil theorem for singular curves, in: Arithmetic, Geometry and Coding Theory}, Luminy, 1993, de Gruyter, Berlin, 1996, 1--7.

\bibitem{BB} D. Bartoli, M. Bonini, {\it A short note on polynomials $f(X)= X+ AX^{1+ q^2(q-1)/4}+ BX^{1+ 3q^2(q-1)/4} \in  \F_{q^2}[X], q$ even,} J. Algebra Appl. 22 (2023) 2350144.

\bibitem{BM} D. Bartoli, M. Timpanella, {\it On trinomials of type $X^{n+ m} (1+ AX^{m (q-1)}+ BX^{n (q-1)})$, $n, m$ odd, over $\F_{q^2}, q= 2^{2s+1}$,} Finite Fields Appl. 72 (2021) 101816.

\bibitem{gupta2022new}R. Gupta, P. Gahlyan, R. K. Sharma, \textit{New classes of permutation trinomials over $\mathbb{F}_{q^3}$}, Finite Fields Appl. 84 (2022) 102110.	

\bibitem{RG} R. Gupta, R. K. Sharma, {\it Some new classes of permutation trinomials over finite fields with even characteristic,} Finite Fields Appl. 41 (2016) 89--96.
	
\bibitem{hou2015permutation}X. Hou, {\it Permutation polynomials over finite fields - a survey of recent advances}, Finite Fields Appl. 32 (2015) 82--119.

\bibitem{hou2015survey}X. Hou, \textit{A survey of permutation binomials and trinomials over finite fields}, Proceedings of the 11th International Conference on Finite Fields and Their Applications, Contemp. Math, 632 (2015) 177--191.

\bibitem{Hou} X. Hou, {\it Applications of the Hasse–Weil bound to permutation polynomials,} Finite Fields Appl. 54 (2018) 113-132.

\bibitem{HC} X. Hou, C. Sze, {\it On a type of permutation rational functions over finite fields,} Finite Fields Appl. 68 (2020) 101758.

\bibitem{NT2017} N. Li, T. Helleseth, \textit{Several classes of permutation trinomials from Niho exponents}, Cryptogr. Commun. 9(6) (2017) 693–705.

\bibitem{NT2019} N. Li, T. Helleseth, \textit{New permutation trinomials from Niho exponents over finite fields with even characteristic}, Cryptogr. Commun. 11 (2019) 129–136.

\bibitem{KLQ2018}K. Li, L. Qu, X. Chen, C.  Li, \textit{Permutation polynomials of the form $cX+ \Tr _{q^l/q}(X^a)$ and permutation trinomials over finite fields with even characteristic}, Cryptogr. Commun. 10 (2018) 531-554.

\bibitem{Li} K. Li, L. Qu, C. Li, S. Fu, {\it New permutation trinomials constructed from fractional polynomials,}
Acta Arith. 183 (2016) 101–116.


\bibitem{li2019survey} N. Li, X. Zeng,  \textit{A survey on the applications of Niho exponents}, Cryptogr. Commun. 11 (2019) 509-548.

\bibitem{ozbudak2023complete}F. Özbudak, B. G. Temür, \textit{Complete characterization of some permutation polynomials of the form $X^r (1+ aX^{{s_1}(q-1)}+ bX^{{s_2}(q-1)})$ over $\mathbb{F}_{q^2}$}, Cryptogr. Commun. 15(4) (2023) 775-793.

\bibitem{Park} Y. H. Park, J. B. Lee, {\it Permutation polynomials and group permutation polynomials}, Bull. Aust. Math. Soc. 63 (2001) 67--74.

\bibitem{tu2018class}Z. Tu, X. Zeng, C. Li, T. Helleseth,  \textit{ A class of new permutation trinomials}, Finite Fields Appl. 50 (2018) 178--195.

\bibitem{Wan} D. Wan, R. Lidl, {\it Permutation polynomials of the form of $f(X^{(q-1/)d})$ and their group structure}, Monatshefte Math. 112 (1991) 149--163.
	
\bibitem{Wang} Q. Wang, {\it Cyclotomic mapping permutation polynomials over finite fields}, in: Sequences, Subsequences, and Consequences, in: Lecture Notes in Comput. Sci., vol. 4893, Springer, Berlin, 2007, pp. 119--128.

\bibitem{wang2024survey}Q. Wang,  \textit{A survey of compositional inverses of permutation polynomials over finite fields}, Des. Codes Cryptogr. (2024) 1--40.

\bibitem{wu2017permutation}D. Wu, P. Yuan, C. Ding, Y. Ma,   \textit{Permutation trinomials over $\mathbb{F}_{2^m}$}, Finite Fields Appl. 46 (2017) 38--56.

\bibitem{Harsh} A. A. Yadav, I. Gupta, H. Singh, A. Yadav, {\it New classes of permutation trinomials of $\F_{2^{2m}}$, } Finite Fields Appl. 96 (2024).

\bibitem{Yadav} A. A. Yadav, H. Singh, I. Gupta, {\it A note on QM equivalence of known permutation polynomials,} Cryptogr. Commun. 17 (2025) 525--540.

\bibitem{ZLF}  Z. Zha, L. Hu, S. Fan, {\it Further results on permutation trinomials over finite fields with even characteristic}, Finite Fields Appl. 45 (2017) 43--52.

\bibitem{LHJ2020} L. Zheng, H. Kan, J. Peng, \textit{Two classes of permutation trinomials with Niho exponents over finite
fields with even characteristic}, Finite Fields Appl. 68 (2020) 101754.

\bibitem{Zieve} M. E. Zieve, {\it Some families of permutation polynomials over finite fields}, Int. J. Number Theory 04 (05) (2008) 851--857.

%\bibitem{M2013}M. E. Zieve, \textit{Permutation polynomials on $\mathbb{F}_{q}$ induced from bijective redei functions on subgroups of the multiplicative group of $\mathbb{F}_{q}$}, (2013) {
%https://doi.org/10.48550/arXiv.1310.0776
%}.
\end{thebibliography}
\end{document}